\documentclass[11pt]{amsart}
\usepackage{amssymb}
\usepackage{lscape}
\usepackage{graphicx}

\newtheorem{theorem}{Theorem}[section]
\newtheorem{prop}{Proposition}
\newtheorem{lemma}{Lemma}[section]
\newtheorem{cor}{Corollary}[section]

\newcommand{\Z}{{\mathbb Z}}
\newcommand{\Q}{{\mathbb Q}}

\newcommand{\od}{{\ell_{o}}}
\newcommand{\ev}{{\ell_{e}}}

\begin{document}

\title[On the smallest number of terms of vanishing sums of units]{On the smallest number of terms of vanishing sums of units in number fields}


\author[Cs. Bert\'ok, K. Gy\H{o}ry, L. Hajdu, A. Schinzel]{Cs. Bert\'ok, K. Gy\H{o}ry, L. Hajdu, A. Schinzel}

\address{Cs. Bert\'ok\newline
\indent Faculty of Informatics\newline
\indent and the MTA-DE Research Group "Equations, Functions and Curves"\newline
\indent of the Hungarian Academy of Sciences and the University of Debrecen\newline
\indent University of Debrecen\newline
\indent H-4002 Debrecen, P.O. Box 400\newline
\indent Hungary}

\email{bertok.csanad@inf.unideb.hu}

\address{K. Gy\H{o}ry, L. Hajdu\newline
\indent Institute of Mathematics\newline
\indent University of Debrecen\newline
\indent H-4002 Debrecen, P.O. Box 400\newline
\indent Hungary}

\email{gyory@science.unideb.hu}
\email{hajdul@science.unideb.hu}

\address{A. Schinzel\newline
\indent Institute of Mathematics\newline
\indent Polish Academy of Sciences\newline
\indent ul. Sniadeckich 8, 00-656, Warszawa\newline
\indent Poland}

\email{schinzel@impan.pl}

\thanks{Research supported in part by the Hungarian Academy of Sciences, the \'UNKP-17-3, New National Excellence Program of the Ministry of Human Capacities, the NKFIH grant K115479 and by the projects EFOP-3.6.1-16-2016-00022, EFOP-3.6.2-16-2017-00015 and EFOP-3.6.3-VEKOP-16-2017-00002, co-financed by the European Union and the European Social Fund.}

\begin{abstract}
Let $K$ be a number field. In the terminology of Nagell a unit $\varepsilon$ of $K$ is called {\it exceptional} if $1-\varepsilon$ is also a unit. The existence of such a unit is equivalent to the fact that the unit equation $\varepsilon_1+\varepsilon_2+\varepsilon_3=0$ is solvable in units $\varepsilon_1,\varepsilon_2,\varepsilon_3$ of $K$. Numerous number fields have exceptional units. They have been investigated by many authors, and they have important applications.

In this paper we deal with a generalization of exceptional units. We are interested in the smallest integer $k$ with $k\geq 3$, denoted by $\ell(K)$, such that the unit equation $\varepsilon_1+\dots+\varepsilon_k=0$ is solvable in units $\varepsilon_1,\dots,\varepsilon_k$ of $K$. If no such $k$ exists, we set $\ell(K)=\infty$. Apart from trivial cases when $\ell(K)=\infty$, we give an explicit upper bound for $\ell(K)$. We obtain several results for $\ell(K)$ in number fields of degree at most $4$, cyclotomic fields and general number fields of given degree. We prove various properties of $\ell(K)$, including its magnitude, parity as well as the cardinality of number fields $K$ with given degree and given odd resp. even value $\ell(K)$.

Finally, as an application, we deal with certain arithmetic graphs, namely we consider the representability of cycles. We conclude the paper by listing some problems and open questions.
\end{abstract}

\date{\today}

\maketitle

\section{Introduction}

Let $K$ be a number field. We are interested in the smallest integer $k$ having the following property:
\begin{equation}
\label{defprop}
\text{there exist units}\ \varepsilon_1,\dots,\varepsilon_k\in K\ \text{such that}\ \varepsilon_1+\dots+\varepsilon_k=0.
\end{equation}
Observe that for any even integer $k=2t$, we have a trivial assertion given by $t\times 1+t\times (-1)=0$. So we shall use the following definitions.

Write $\od(K)$ for the smallest odd $k\geq 3$ for which \eqref{defprop} is valid. Further, let $\ev(K)$ be the smallest even $k\geq 4$ for which \eqref{defprop} is valid, such that the sum appearing in \eqref{defprop} has no proper vanishing subsum. If no appropriate $k$ exists at all, then we set $\od(K)=\infty$ or $\ev(K)=\infty$, respectively. We put $\ell(K)=\min(\od(K),\ev(K))$.

Before proceeding further, we make a trivial observation. First note that if $k=\od(K)$, then the sum of units appearing in \eqref{defprop} has no proper vanishing subsum. Indeed, otherwise we would have a proper vanishing subsum with an odd number of terms, contradicting the minimality of $k=\od(K)$.

The above notions can be generalized to orders of number fields. Let $\mathcal{O}$ be an order of a number field $K$. Then we can define $\od(\mathcal{O}),\ev(\mathcal{O}),\ell(\mathcal{O})$ in the
obvious way. Note that if $\mathcal{O}$ is the maximal order of $K$, then we clearly have $\od(\mathcal{O})=\od(K)$, $\ev(\mathcal{O})=\ev(K)$, $\ell(\mathcal{O})=\ell(K)$.

In this paper we obtain several results concerning $\ell(K),\od(K),\ev(K)$ and $\ell(\mathcal{O}),\od(\mathcal{O}),\ev(\mathcal{O})$. We show among other things that $\ell(K)$ is finite for any number field $K$, apart from the cases where $K=\Q$ or $K$ is an imaginary quadratic field. Further, we prove that for any integer $k\geq 3$ there exists an order of a real quadratic number field with $\ell(\mathcal{O})=k$, and also a complex cubic number field $K$ with $\ell(K)=k$ - in the latter case excluding values $k$ of the form $k=4t^4-4t+2$. On the other hand, we show that for each $k$, there are only finitely many quadratic fields, complex cubic fields and (up to certain completely described exceptions) totally complex quartic fields with $\ell(K)\leq k$, and all these number fields can be effectively determined. Furthermore, it is shown that for any number field $K$ different from $\Q$ and the imaginary quadratic fields we have $\ev(K)<\infty$. Finally, we prove that for $d\geq 3$ there are infinitely many number fields $K$ of degree $d$ with $\ev(K)=4$, and for $d\geq 2$ there are infinitely many number fields $K$ of degree $d$ with $\od(K)=\infty$.

We give some applications of our results to certain arithmetic graphs, more precisely to graphs having vertices from the set of integers of $K$, in which two vertices $\alpha,\beta$ are connected by an edge if and only if $\alpha-\beta$ is a unit in $K$. We mention that Gy\H{o}ry has several results about and applications of such graphs (see e.g. \cite{gy2} and the references given there), and recently Gy\H{o}ry, Hajdu, Tijdeman \cite{gyht1,gyht2} and Ruzsa \cite{r} made a systematic study of the representability of such graphs. Our results allow us to extend some results from the mentioned papers, concerning representations of cycles.

Clearly, the existence of units appearing in \eqref{defprop} means that the unit equation
\begin{equation}
\label{gyeq1}
\varepsilon_1+\dots+\varepsilon_k=0
\end{equation}
has a solution in units $\varepsilon_1,\dots,\varepsilon_k$ of $K$ such that the left hand side has no proper vanishing subsum. For $k=3$, the solvability of \eqref{gyeq1} is equivalent to the existence of a unit $\varepsilon$, called exceptional unit, see Nagell \cite{n4}, such that $1-\varepsilon$ is also a unit. Obviously, we have $\ell(K)=\od(K)=3$ if and only if $K$ contains an exceptional unit. There is an extremely rich literature on unit equations of the form \eqref{gyeq1}. For given $k\geq 3$, there are results stating the finiteness of the number of solutions up to a proportional factor. Further, there are explicit upper bounds for the number of solutions and, for $k=3$, even for the size of the solutions. Moreover, for $k=3$ and for some special number fields $K$, all the solutions have been determined. Many books and survey papers deal with these equations, their generalizations and various applications; see e.g. Lang \cite{lang}, Gy\H{o}ry \cite{gynew1,gynew2}, Evertse \cite{evnew}, Mason \cite{mason}, Shorey and Tijdeman \cite{st}, Evertse, Gy\H{o}ry, Stewart and Tijdeman \cite{egyst}, Schmidt \cite{schmidt}, Smart \cite{smart}, Evertse and Gy\H{o}ry \cite{egy} and the references given there.

We organize our paper as follows. First we present our main results, followed by their proofs. After that we give applications to arithmetic graphs. We conclude the paper with some open problems.

\section{Main results}

In this section we present our main results. We split them into two parts: first we provide statements concerning the parameters $\ell(K),\od(K),\ev(K)$ and $\ell({\mathcal O}),\od({\mathcal O}),\ev({\mathcal O})$. Then we give results concerning so-called odd and even units, since they play an important role in our proofs.

\subsection{Results concerning $\ell(K),\od(K),\ev(K)$ and $\ell({\mathcal O}),\od({\mathcal O}),\ev({\mathcal O})$}

Our first theorem is a simple, but important statement.

\begin{theorem}
\label{thm1}
For any number field $K$ different from $\Q$ and the imaginary quadratic fields, $\ell(K)$ is finite. Further,
$$
\ell(K)\leq 2(d+1)\exp\{c R_K\},
$$
where
$$
c=
\begin{cases}
1/d,&\text{if}\ $r=1$,\\
29e\sqrt{r-1}\cdot r!(\log d),&\text{if}\ r\geq 2.
\end{cases}
$$
Here $r,d$ and $R_K$ denote the unit rank, the degree and the regulator of $K$, respectively.
\end{theorem}

We note that
$$
R_K\leq |D_K|^{1/2} (\log^* |D_K|)^{d-1},
$$
where $D_K$ denotes the discriminant of $K$, and $\log^*(x)=\max\{\log x,1\}$. This is an improvement of an inequality of Landau \cite{lan}; see (59) in Gy\H{o}ry and Yu \cite{gyy}.

\vskip.2cm

\noindent
{\bf Remark.} Obviously, $\ell(K)=\infty$ for $K=\Q$ and the same is true for all imaginary quadratic fields (including the Gaussian field $\Q(i)$), except for $K=\Q(\sqrt{-3})$. In the latter case we have $\ell(K)=3$.

We also mention that a statement similar to Theorem \ref{thm1} could be formulated for orders $\mathcal{O}$ of number fields, as well.

\vskip.2cm

Our next result shows that $\ell(\mathcal{O})$ can be an arbitrary integer $k\geq 3$.

\begin{theorem}
\label{thm2}
For any $k\geq 3$ there exists an order $\mathcal{O}$ of some number field $K$ with $\ell(\mathcal{O})=k$. In fact,
$\mathcal{O}$ can be chosen as an order of a real quadratic number field.
\end{theorem}

Our next result shows that apart from the values $k$ taken by a particular quartic polynomial, $\ell(K)$ can also be an arbitrary integer $k\geq 3$.

\begin{theorem}
\label{thm3}
For any $k\geq 3$ which is not of the form $4t^4-4t+2$ $(t\in\Z\setminus\{0,1\})$ there exists a number field $K$ with $\ell(K)=k$. In fact, one can choose $K$ to be a complex cubic number field.
\end{theorem}

We are sure that the above theorem is valid for all values of $k$. So we propose the following

\vskip.2cm

\noindent
{\bf Conjecture.}
For any $k\geq 3$ there exists a number field $K$ with $\ell(K)=k$.

\vskip.2cm

We provide some numerical results to support our conjecture.

\begin{prop}
\label{propconj}
Let $K_t=\Q(\alpha_t)$, where $\alpha_t$ is a root of the polynomial $x^3+x^2+(4t^4-4t-1)x+1$ for $t\in\{-20,\dots,-1\}\cup\{2,\dots,20\}$. Then we have $\ell(K_t)=4t^4-4t+2$.
\end{prop}

Our next theorem shows that under some restrictions, for any $k$, there are only finitely many number fields $K$ with $\ell(K)\leq k$. Clearly, some restriction is needed to obtain such a result: for example, if $\varepsilon$ is a root of the polynomial $x^n+x+1$ with $n\geq 2$, then for the number field $K=\Q(\varepsilon)$ we obviously have $\ell(K)=3$. In what follows, we write $\zeta_n$ for a primitive root of unity of order $n$.

\begin{theorem}
\label{thm4} For any $k\geq 3$, there are only finitely many quadratic fields, complex cubic number fields and totally complex quartic number fields $K$ with $\ell(K)\leq k$, in the latter case assuming that $K$ does not have a real quadratic subfield and $\zeta_3\notin K$, and all such fields can be effectively determined.
\end{theorem}

\noindent
{\bf Remark.} There are infinitely many totally complex quartic fields $K$ having a real quadratic subfield $L$. As for all such fields $K$ we have $\ell(K)\leq\ell(L)$, the above statement is not valid for them. Similarly, there are infinitely many totally complex quartic fields $K$ with $\zeta_3\in K$, and hence with $\ell(K)=3$. Thus they also have to be excluded from Theorem \ref{thm4}.

\vskip.2cm

The value of $\od(K)$ can be infinite in non-trivial cases (i.e. excluding $\Q$ and the imaginary quadratic fields) as well.

\begin{theorem}
\label{thm5}
Let $d\geq 2$. There are infinitely many number fields $K$ of degree $d$ with $\od(K)=\infty$.
\end{theorem}

Our final result in this subsection shows that $\ev(K)$ can take its minimal value (that is $4$) for infinitely many number fields, having any prescribed degree $\geq 3$. Note that in view of Theorem \ref{thm4}, the case $d=2$ has to be excluded, so our statement is best possible in this respect.

\begin{theorem}
\label{thm6}
Let $d\geq 3$. There are infinitely many number fields $K$ of degree $d$ with $\ev(K)=4$.
\end{theorem}

\subsection{Results concerning odd and even units}

In this subsection we investigate the existence of odd and even units in a number field $K$. This is an important question from our viewpoint: as we shall see soon, if $K$ contains an odd unit then $\od(K)$ is finite, and, similarly, if $K$ contains an even unit then $\ev(K)$ is finite.

We need a little preparation. For any integer polynomial
$$
g(x)=b_nx^n+b_{n-1}x^{n-1}+\dots+b_1x+b_0
$$
write
$$
L(g)=|b_n|+|b_{n-1}|+\dots+|b_1|+|b_0|
$$
for the length of $g(x)$. The properties of lengths of polynomials have been studied by several authors; see e.g. \cite{fs,s2,s3} and the references given there.

We call an algebraic integer $\alpha$ even, if $L(f)$ is even, where $f(x)$ is the minimal monic polynomial of $\alpha$ (over $\Q$); otherwise $\alpha$ is odd. Observe that $\alpha$ is even if and only if $f(1)$ is even.

Let $\varepsilon$ be a unit of $K$, different from the roots of unity, and let $f(x)=x^n+a_{n-1}x^{n-1}+\dots+a_1x+a_0$ be the minimal monic polynomial of $\varepsilon$. Observe that then the equation
$$
\varepsilon^n+a_{n-1}\varepsilon^{n-1}+\dots+a_1\varepsilon+a_0=0
$$
shows that \eqref{defprop} is satisfied in $K$ with $k=L(f)$ terms, and also that there cannot be proper vanishing subsums of the left hand side. In particular, we have that if $\varepsilon$ is odd then $\od(K)<\infty$, and if $\varepsilon$ is even then $\ev(K)<\infty$. In what follows, this observation will be frequently used.

The next theorem shows that excluding the trivial cases, every number field contains a non-trivial even unit.

\begin{theorem}
\label{newcor1}
Every number field $K$ different from $\Q$ and the imaginary quadratic fields, contains an even unit different from $\pm 1$. In particular, we have $\ev(K)<\infty$.
\end{theorem}

As it was mentioned in the Remark after Theorem \ref{thm1}, for $K=\Q$ and the imaginary quadratic fields with the exception of $\Q(\zeta_3)$, we have $\ell(K)=\infty$. Hence $\ev(K)=\infty$ is also valid for these fields. Further, it is easy to check that $\ev(\Q(\zeta_3))=\infty$, too.

Theorems \ref{thm5} and \ref{newcor1} imply that for $d\geq 3$ and for $d=2$ with $K$ quadratic real, $\ev(K)<\infty$ holds, and there are infinitely many number fields $K$ of degree $d$ with $\od(K)=\infty$.

For $d\leq 4$, we have the following more explicit result.

\begin{theorem}
\label{newcor2}
Let $K$ be a real quadratic, a complex cubic or a totally complex quartic field; in the latter case assume that $K$ does not contain roots of unity different from $\pm 1$. Suppose that $K$ has a fundamental unit which is even. Then all units of $K$ are even. In particular, in these cases we have $\od(K)=\infty$.
\end{theorem}

Our final result in this section shows that in general, cyclotomic fields contain odd units.

\begin{theorem}
\label{newthm4}
In every cyclotomic field $K=\Q(\zeta_n)$ except $n\mid 4$ there exists an odd unit. In particular, in these number fields we have $\od(K)<\infty$ and $\ev(K)<\infty$.
\end{theorem}

By our previous remarks, for $n=1,2,4$ we have $\ell(K)=\infty$.

\section{Lemmas, auxiliary results and proofs of our main results}

We start with the proofs of our theorems concerning odd and even units. For this, we need several lemmas.

\begin{lemma}
\label{lemnewsc}
Let $F$ be the minimal monic polynomial of a unit $\varepsilon$ over $\Q$ and $n$ a positive integer. Then $\varepsilon^n$ is even if $\prod_{i=0}^{n-1} F(\zeta_n^i)$ is even, where $\zeta_n$ is a primitive root of unity of order $n$.
\end{lemma}

\begin{proof} Let $G$ be the minimal monic polynomial of $\varepsilon^n$ over $\Q$. Since $G(\varepsilon^n)=0$, we have $F(x)\mid G(x^n)$. It follows that for every $i=0,\dots,n-1$ we have $F(\zeta_n^ix)\mid G(x^n)$ whence $\prod_{i=0}^{n-1} F(\zeta_n^ix)\mid G(x^n)^n$. If the assumption of the lemma holds, then $G(1)$ is even, thus $\varepsilon^n$ is even.
\end{proof}

Let $\phi$ be the canonical map of $\Z[x]$ onto ${\mathbb F}_2[x]$.

\begin{lemma}
\label{newthm1}
Let $\varepsilon$ be a unit of a number field $K$ and
$F$ its minimal polynomial over $\Q$. Then $\varepsilon^{2^n-1}$ is even, if $n$ is the degree of an irreducible factor over ${\mathbb F}_2$ of $\phi(F)$.
\end{lemma}

\begin{proof} By the theory of finite fields we have
$$
x^{2^n}-x=\prod f(x),
$$
where the product on the right hand side is taken over all distinct irreducible polynomials over ${\mathbb F}_2$ whose degree divides $n$. Hence, over ${\mathbb F}_2$ we have
$$
\phi(\Phi_{2^n-1}(x))=\prod f(x),
$$
where $\Phi_{m}$ denotes the cyclotomic polynomial of order $m$, and the product on the right hand side is now taken over all distinct irreducible polynomials over ${\mathbb F}_2$ of degree $n$. By Dedekind's theorem on congruences we have
$$
(2)=\prod(G_f(\zeta_{2^n-1}),2)\ \ \ \text{with}\ G_f\in\Z[x],\ \phi(G_f)=f,
$$
where the product on the right hand side is taken as before, and the ideals are prime. Write ${\mathcal P}_{f,n}=(G_f(\zeta_{2^n-1}),2)$. It follows that if $f\mid F$  over ${\mathbb F}_2$ with deg$(f)=n$, then we have
$$
\prod\limits_{j=1}^{2^n-1} F(\zeta_{2^n-1}^j)\equiv N_{\Q(\zeta_{2^n-1})/\Q}(F(\zeta_{2^n-1}))\equiv 0\pmod{N_{\Q(\zeta_{2^n-1})/\Q}({\mathcal P}_{f,n})},
$$
whence the congruence also holds modulo $2$. This shows, by Lemma \ref{lemnewsc}, that $\varepsilon^{2^n-1}$ is even, and the statement follows.
\end{proof}

\begin{proof}[Proof of Theorem \ref{newcor1}] The statement is an immediate consequence of Lemma \ref{newthm1}.
\end{proof}

To prove Theorem \ref{newcor2}, we need the following

\begin{lemma}
\label{newthm2}
Let $K$ be a real quadratic, a complex cubic or a totally complex quartic field; in the latter case assume that $K$ does not contain roots of unity different from $\pm 1$. Let $\varepsilon$ be a fundamental unit of $K$. Suppose that $\varepsilon$ is even. Then if $\varepsilon_1+\dots+\varepsilon_k=0$ holds for some units $\varepsilon_1,\dots,\varepsilon_k$ of $K$ then $k$ is even.
\end{lemma}

\begin{proof} Suppose to the contrary that with some odd $k$, we have an equality of the form $\varepsilon_1+\dots+\varepsilon_k=0$. Let $f(x)$ be the minimal monic polynomial of $\varepsilon$ over $\Z$. By multiplying the equation by an appropriate power of $\varepsilon$ (in view of that the unit rank of $K$ is one and $K$ contains no roots of unity different from $\pm 1$) we get an equation of the form $h(\varepsilon)=0$, where $h\in\Z[x]$. Dividing $h$ by an appropriate integer if necessary, we may further assume that it is primitive. Then, by the Gauss lemma we easily deduce that $h(x)=f(x)g(x)$ holds, where $g(x)\in\Z[x]$. However, since $L(h)$ is odd and $L(f)$ is even, it yields a contradiction. Hence the lemma follows.
\end{proof}

\begin{proof}[Proof of Theorem \ref{newcor2}] The statement is an immediate consequence of Lemma \ref{newthm2}.
\end{proof}

Since we find it of independent interest, now we show that a statement similar to Lemma \ref{newthm2} is true for totally real cubic fields (having already unit rank $2$).

\begin{prop}
\label{newthm3}
Let $K$ be a totally real cubic field. Suppose that $K$ has a system of fundamental units which are even. Then all units of $K$ are even.
\end{prop}

\begin{proof}
Write $\varepsilon,\eta$ for a system of fundamental units of $K$. Since $\varepsilon$ is even, either $\varepsilon^3+\varepsilon^2+\varepsilon+1\equiv 0\pmod{2}$, or $\varepsilon^3+1\equiv 0\pmod{2}$, and the same is valid with $\eta$ in place of $\varepsilon$. If $\varepsilon^3+1\equiv \eta^3+1\pmod{2}$ then all units of $K$ are even. Indeed, if $\nu=\pm \varepsilon^m\eta^n$ $(m,n\in\Z)$ would be an odd unit of $K$ with minimal monic polynomial $x^3+ax^2+bx\pm 1$, then we would have
$$
a+b\equiv 1\pmod{2}\ \ \ \text{and}\ \ \ a(\varepsilon^m\eta^n)^2+b(\varepsilon^m\eta^n)\equiv 0\pmod{2},
$$
which is impossible.

Thus without loss of generality we may assume that
$$
\varepsilon^3+\varepsilon^2+\varepsilon+1\equiv
0\pmod{2}.
$$
We have five cases, according to the splitting of the prime $2$ (the principal ideal $(2)$) in $K$:
\begin{itemize}
\item $(2)={\mathcal P}_1$, ${\mathcal P}_1$ is a prime ideal,
\item $(2)={\mathcal P}_1{\mathcal P}_2$, ${\mathcal P}_i$ is a prime ideal of degree $i$ $(i=1,2)$,
\item $(2)={\mathcal P}_1{\mathcal P}_2{\mathcal P}_3$, the ${\mathcal P}_i$ are distinct prime ideals $(i=1,2,3)$,
\item $(2)={\mathcal P}_1^2{\mathcal P}_2$, the ${\mathcal P}_i$ are distinct prime ideals $(i=1,2)$,
\item $(2)={\mathcal P}_1^3$, ${\mathcal P}_1$ is a prime ideal.
\end{itemize}
Observe that in all cases, by
$$
\varepsilon^3+\varepsilon^2+\varepsilon+1\equiv (\varepsilon+1)^3\equiv 0\pmod{2},
$$
we obtain $\varepsilon\equiv 1\pmod{{\mathcal P}_1}$.
If we had also $\eta\equiv 1\pmod{{\mathcal P}_1}$, then
$\pm \varepsilon^m\eta^n\equiv 1\pmod{{\mathcal P}_1}$ $(m,n\in\Z)$ would follow, showing that
\begin{equation}
\label{ujneweq}
a+b\equiv 1\pmod{2}\ \ \ \text{and}\ \ \ a(\varepsilon^m\eta^n)^2+b(\varepsilon^m\eta^n)\equiv 0\pmod{{\mathcal P}_1}
\end{equation}
is impossible. It remains to check the case
$$
\varepsilon\equiv1\pmod{{\mathcal P}_1}\ \ \ \text{and}\ \ \
\eta^3\equiv 1\pmod{2}.
$$
However, as one can easily check, \eqref{ujneweq} is also impossible in this case. Hence our statement follows.
\end{proof}

\begin{proof}[Proof of Theorem \ref{newthm4}]
If $n\neq 2^\alpha$, one can take $\varepsilon=\zeta_n$. Indeed, the minimal monic polynomial of $\zeta_n$ is $\Phi_n(x)$ and we have
$$
\Phi_n(1)\equiv 1\pmod{2}.
$$
If $n=2^\alpha$ $(\alpha\geq 3)$, one can take
$$
\varepsilon=1+\zeta_8+\zeta_8^2=\frac{\zeta_8^3-1}{\zeta_8-1}.
$$
Indeed, the minimal monic polynomial of $\varepsilon$ is $x^4+14x^3+5x^2+2x+1$, and the theorem follows.
\end{proof}

Now we turn to the proofs of our theorems concerning $\ell(K),\od(K),\ev(K)$. In fact, the proof of Theorem \ref{thm1} is based upon a simple observation.

\begin{proof}[Proof of Theorem \ref{thm1}]
Let $\varepsilon$ be a unit of $K$, different from $1$ and $-1$. Write $f(x)=x^n+a_{n-1}x^{n-1}+\dots+a_1x+a_0\in\Z[x]$ for the minimal monic polynomial of $\varepsilon$. Then the equality $f(\varepsilon)=0$ can be considered as an equation of the form \eqref{defprop}, with $k:=1+|a_{n-1}|+\dots+|a_1|+|a_0|$ terms on the left hand side. Since $f(x)$ is the minimal monic polynomial of $\varepsilon$, it is obvious that this equation has no vanishing subsums. This proves that $\ell(K)\leq k$. Since $n\leq d$ and $k\leq (d+1)H(f)$, where $H(f)$ denotes the height (i.e. the maximum absolute value of the coefficients) of $f$, it suffices to to give an upper bound for the height of the minimal monic polynomial of an appropriate unit $\varepsilon$ of $K$.

It follows from Proposition 4.3.9 in Evertse and Gy\H{o}ry \cite{egy}, an improvement of a classical result of Siegel \cite{si2}, that there is a unit $\varepsilon$ in $K$ such that $h(\varepsilon)\leq cR_K$ with the constant specified in Theorem \ref{thm1}. Here $h(\varepsilon)$ denotes the absolute logarithmic height of $\varepsilon$. But by (1.9.3) of Evertse and Gy\H{o}ry \cite{egy}, the height of the minimal monic polynomial of $\varepsilon$ is at most $2\exp\{h(\varepsilon)\}$, hence the claimed upper bound for $\ell(K)$ follows.
\end{proof}

To prove Theorem \ref{thm2}, we need two lemmas. The first one is due to Louboutin \cite{lo}.

\begin{lemma}
\label{lemlo}
Let $\varepsilon>1$ be a real quadratic unit. Then $\varepsilon$ is the fundamental unit of the quadratic order $\Z[\varepsilon]$, with the sole exception of $\varepsilon =(3+\sqrt{5})/2$.
\end{lemma}

\begin{proof} The statement is an immediate consequence of Theorem 1 of \cite{lo}.
\end{proof}

The next lemma shows that in case of some quadratic and cubic polynomials $f(x)\in\Z[x]$ of special shape, $L(fg)\geq L(f)$ holds for all $g(x)\in\Z[x]$ which is not identically zero.

\begin{lemma}
\label{lemlfg}
Let $a$ be a positive integer, and $f(x)$ be one of the polynomials $x^2-ax-1$, $x^3+ax+1$, $x^3+x^2+ax+1$; in the latter case assume further that $a\geq 3$. Then for any $g(x)\in\Z[x]$ not identically zero, we have $L(fg)\geq L(f)$.
\end{lemma}

\begin{proof} Let $a$ be a positive integer, and $g(x)=b_nx^n+\dots+b_1x+b_0$ with $n\geq 0$ and $b_n,\dots,b_0\in \Z$, $b_n\neq 0$. Clearly, we may assume that $n\geq 1$, $b_n>0$ and $b_0\neq 0$, whence $L(g)\geq 2$. Further, we put $h(x)=f(x)g(x)$.

First let $f(x)=x^2-ax-1$. Then we have
$$
h(x)=c_{n+2}x^{n+2}+\dots+c_1x+c_0,
$$
with
$$
c_{n+2}=b_n,\ c_{n+1}=b_{n-1}-ab_n,\ c_1=-ab_0-b_1,\ c_0=-b_0
$$
and
$$
c_i=b_{i-2}-ab_{i-1}-b_i\ (i=2,\dots, n).
$$
Hence we get
$$
L(h)=\sum\limits_{i=0}^{n+2} |c_i|\geq |b_n|+|b_0|+aL(g)-\sum\limits_{i=0}^{n-1} |b_i|-\sum\limits_{i=1}^n |b_i|\geq
(a-2)L(g)+4.
$$
As $L(g)\geq 2$ and $L(f)=a+2$, this implies our claim for $a\geq 2$. If $a=1$ then $L(f)=3$, and we are done unless $L(h)=2$, that is, $h(x)=x^{n+2}\pm 1$. However, then $f(x)\nmid h(x)$, which is a contradiction, proving our claim in this case.

Assume now that $f(x)=x^3+ax+1$. Then we have
$$
h(x)=c_{n+3}x^{n+3}+\dots+c_1x+c_0,
$$
with
$$
c_{n+3}=b_n,\ \ \ c_{n+2}=b_{n-1},\ \ \ c_{n+1}=b_{n-2}+ab_n,
$$
$$
c_2=ab_1+b_2,\ \ \ c_1=ab_0+b_1,\ \ \ c_0=b_0
$$
and
$$
c_i=b_{i-3}+ab_{i-1}+b_i\ (i=3,\dots, n).
$$
Similarly to the case $f(x)=x^2-ax-1$, we get
$$
L(h)\geq |b_n|+|b_{n-1}|+|b_0|+aL(g)-\sum\limits_{i=0}^{n-2} |b_i|-\sum\limits_{i=1}^n |b_i|\geq
(a-2)L(g)+4.
$$
This gives that the statement is valid for $a\geq 2$. For $a=1$, $L(h)<L(f)$ would imply $x^3+x+1\mid x^{n+3}\pm 1$, which does not hold. Hence the lemma follows also in this case.

Finally, let $f(x)=x^3+x^2+ax+1$. Then we can write
$$
h(x)=c_{n+3}x^{n+3}+\dots+c_1x+c_0,
$$
with
$$
c_{n+3}=b_n,\ \ \ c_{n+2}=b_{n-1}+b_n,\ \ \ c_{n+1}=b_{n-2}+b_{n-1}+ab_n,
$$
$$
c_2=b_0+ab_1+b_2,\ \ \ c_1=ab_0+b_1,\ \ \ c_0=b_0
$$
and
$$
c_i=b_{i-3}+b_{i-2}+ab_{i-1}+b_i\ (i=3,\dots, n).
$$
Similarly to the case $f(x)=x^2-ax-1$, we get
$$
L(h)\geq |b_n|+|b_{n-1}+b_n|+|b_0|+aL(g)-\sum\limits_{i=0}^{n-2} |b_i|-\sum\limits_{i=0}^{n-1} |b_i|-\sum\limits_{i=1}^n |b_i|\geq
$$
$$
\geq (a-3)L(g)+6.
$$
As $a\geq 3$, $L(f)=a+3$ and $L(g)\geq 2$, this gives $L(h)\geq L(f)$, and the lemma follows.
\end{proof}

\begin{proof}[Proof of Theorem \ref{thm2}]
Let $k\geq 3$. Let $\varepsilon$ be a root of the polynomial $f(x)=x^2-(k-2)x-1$, and set $\mathcal{O}=\mathbb{Z}[\varepsilon]$. By Lemma \ref{lemlo} we know that $\varepsilon$ is a fundamental unit of $\mathcal{O}$. Then, in the same way as in the proof of Lemma \ref{newthm2}, we see that all vanishing sums of units in $K=\Q(\varepsilon)$ are obtained from the integer polynomial multiples $h(x)$ of $f(x)$. Now by Lemma \ref{lemlfg} we get that for all such $h(x)$, $L(h)\geq L(f)=k$ holds. This implies the statement.
\end{proof}

To prove Theorem \ref{thm3}, we need a result concerning cubic factors of certain special trinomials. For theorems on the reducibility of general trinomials, see e.g. \cite{s1} and the corresponding chapter of \cite{selecta}.

\begin{lemma}
\label{lembr}
Let $m,A,E$ be integers with $m\geq 2$ and $E\in\{-1,1\}$. Suppose that $x^{3m}+Ax^m+E$ has an irreducible cubic factor in $\Z[x]$. Then one of the following cases occurs:
\begin{itemize}
\item[(i)] $m=11$, $A=67$ and $E=1$, when $x^3+x+1$ is the only cubic factor,
\item[(ii)] $m=4$, $A=1040$ and $E=-1$,
\item[(iii)] $m=2$, $E=-1$ and $A$ is of the form $A=4t^4-4t$ $(t\in\Z\setminus\{0,1\})$.
\end{itemize}
\end{lemma}

\begin{proof} The statement is an immediate consequence of the Theorem in Tverberg \cite{tv}. Note that this result of Tverberg is an extension of the Theorem in Bremner \cite{br}, where only the case $E=1$ was considered. It is easy to check (e.g. by Magma \cite{magma}) that the only cubic factor of the polynomial $x^{33}+67x^{11}+1$ is $x^3+x+1$.
\end{proof}

Now we can give the

\begin{proof}[Proof of Theorem \ref{thm3}] For given $k$ not of the form $4t^4-4t+2$, take $A=k-2$ and consider the polynomial
$f(x)=x^3+Ax+1$. As one can easily check, $f(x)$ (in view of $A\geq 1$) is irreducible over $\Q$, and has one real and two complex roots. Let $\varepsilon$ be a root of $f(x)$, and put $K=\Q(\varepsilon)$. Write $\varepsilon=\pm \eta^m$ with some $m\geq 2$, where $\eta$ is an appropriately chosen fundamental unit of $K$. Let $h(x)$ be the minimal monic polynomial of $\eta$. It is easy to see that $h(x)$ divides one of the polynomials $x^{3m}+Ax^m\pm 1$ in $\Z[x]$. Noting that as $1040=4(-4)^4-4(-4)$ we have $A\neq 1040$, by Lemma \ref{lembr} we obtain that if $A\neq 67$ then $m=1$ holds.

We conclude that if $A$ is not of the form $4t^4-4t$ $(t\in\Z\setminus\{0,1\})$, then $\varepsilon$ is a fundamental unit of $K$, unless $A=67$. So in the cases where $A\neq 67$, just as before, we get that any vanishing sum of units in $K$ comes from a multiple of $f(x)$. However, by Lemma \ref{lemlfg} we obtain that the number of the terms in any such sum is at least $L(f)=A+2=k$, and the theorem follows in these cases.

Hence we are left with $k-2=A=67$. In this case consider the polynomial $f(x)=x^3+x^2+66x+1$. A simple check by Magma \cite{magma} shows that this polynomial is irreducible, has one real and two complex roots. Further, taking a root $\varepsilon$ of $f(x)$, $\varepsilon$ is a fundamental unit of $K=\Q(\varepsilon)$. By Lemma \ref{lemlfg} we get that for any $g\in\Z[x]$ which is not identically zero, we have $L(fg)\geq L(f)=69$. This in the same way as before shows that $\ell(K)=69$, and the theorem follows.
\end{proof}

Now we give the proof of Proposition \ref{propconj}.

\begin{proof}[Proof of Proposition \ref{propconj}] A simple calculation with Magma \cite{magma} shows that $\alpha_t$ is a fundamental unit of $K_t$ for the values of $t$ under consideration. Hence following the usual argument, the statement follows by Lemma \ref{lemlfg}.
\end{proof}

To prove Theorem \ref{thm4} we need the following lemma, essentially due to Mignotte \cite{mi}. It provides a weaker, but much more general lower bound for $L(fg)$ than Lemma \ref{lemlfg}.

\begin{lemma}
\label{mignotte}
Let $f\in\Z[x]$ of degree $n\geq 0$. Then for any $g\in\Z[x]$ which is not identically zero, we have $L(fg)\geq 2^{-n}L(f)$.
\end{lemma}

\begin{proof} Write
$$
f(x)=a_nx^n+\dots+a_1x+a_0\ \ \ \text{and}\ \ \
f(x)g(x)=b_sx^s+\dots+b_1x+b_0.
$$
Theorem 2 of \cite{mi} gives
$$
|a_i|\leq {n\choose i} \sqrt{\sum\limits_{j=0}^s b_j^2}\ \ \ (i=0,\dots,n).
$$
Thus
$$
L(f)=|a_n|+\dots+|a_0|\leq 2^n \sqrt{\sum\limits_{j=0}^s b_j^2}\leq 2^n L(fg),
$$
and the statement follows.
\end{proof}

The last assertion we need in the proof of Theorem \ref{thm4} concerns lengths of polynomials $g(x)$ such that $L(fg)$ is "small" for a given $f(x)$.

\begin{lemma}
\label{filaseta}
Let $f(x)\in\Z[x]$ having no cyclotomic factors, and let $N\geq 1$. Then there exists an effectively computable constant $C(L(f),N)$ depending only on $L(f)$ and $N$ such that for any $g(x)\in\Z[x]$ with $L(fg)\leq N$ we have $L(g)\leq C(L(f),N)$. Further, at least one such $g$ satisfies ${\rm deg}(g)\leq C(L(f),N)({\rm deg}(f)+1)$.
\end{lemma}

\begin{proof} The first part of the statement is an immediate consequence of Theorem 1 in \cite{fs}. Note that in \cite{fs} in place of the length the authors work with another norm, however, it is easy to reformulate their result for $L(g)$. Further, $C(L(f),N)$ is not claimed to be effective in \cite{fs}, but following the argument there, one can easily see that this constant is effectively computable, indeed. (See also Theorem 3 in \cite{fs}, where in a special case a $C(L(f),N)$ is explicitly given.)

To prove the second statement concerning the degree of $g$, observe the following. Writing $n={\rm deg}(f)$, $m={\rm deg}(g)$ and $g(x)=\sum\limits_{i=0}^m b_ix^i$, if $n+1$ consecutive coefficients of $g$, say $b_i,\dots,b_{i+n}$ are all zero, then clearly $L(fg)=L(fg^*)$ with
$g^*(x)=\sum\limits_{j=0}^{n+i-1}b_jx^j+\sum\limits_{j=n+i}^{m-1}b_{j+1}x^j$.
This shows that if $L(fg)\leq N$ with $L(g)\leq C(L(f),N)$, then starting from $g$, we can construct a $g_0(x)\in\Z[x]$ such that $L(fg_0)\leq N$, $L(g_0)\leq C(L(f),N)$ and there are at most $n$ consecutive zeros among the coefficients of $g_0$. Hence the statement follows.
\end{proof}

\begin{proof}[Proof of Theorem \ref{thm4}] Let $k\geq 3$ be fixed, and suppose that $K$ is an algebraic number field as in the statement, with $\ell(K)\leq k$. Let $\varepsilon$ be a fundamental unit of $K$, and write $f(x)=x^n+a_{n-1}x^{n-1}\dots+a_1x+a_0\in\Z[x]$ for its minimal monic polynomial. Note that here $n\in\{2,3,4\}$ and $a_0\in\{-1,1\}$. Further, since by our assumption $K$ has no real quadratic subfields and $\Q(\varepsilon)$ cannot be an imaginary quadratic field, we also have $K=\Q(\varepsilon)$.

Let $\varepsilon_1,\dots,\varepsilon_k$ be units in $K$ with
\begin{equation}
\label{ujegy}
\varepsilon_1+\dots+\varepsilon_k=0.
\end{equation}
Assume first that $K$ does not contain roots of unity different from $\pm 1$. By the usual argument, since by our assumption $K$ does not contain any roots of unity different from $\pm 1$, this gives $h(\varepsilon)=0$ with $h\in\Z[x]$ such that $L(h)=k$. Hence for some $g\in\Z[x]$ not identically zero, we have $L(fg)=k$. This by Lemma \ref{mignotte} yields $L(f)\leq 16k$. So as $K=\Q(\varepsilon)$, there are only finitely many such $K$. Checking all the possibilities with $L(f)\leq 16k$, in view of Lemma \ref{filaseta} these number fields can be effectively determined.

Suppose next that $K$ contains some root of unity different from $\pm 1$. Then $K$ contains a primitive $m$-th root of unity $\eta$ with some $m\geq 3$. As we have $\varphi(m)\leq 4$, we get that
$$
m\in\{3,4,5,6,8,10,12\}.
$$
If $m$ is one of $3,6,12$, then $\zeta_3\in K$, which is excluded. If $m$ is $5$ or $10$, then $K$ is defined by the polynomial $x^4+x^3+x^2+x+1$. However, then (as one can readily check e.g. by Magma \cite{magma}) $K$ has $\Q(\sqrt{5})$ as a subfield, which is excluded again. If $m=8$, then $K$ is defined by $x^4+1$, and using again Magma, we see that $\Q(\sqrt{2})$ is a subfield of $K$, which is also excluded. So we are left with the only possibility $m=4$, and the roots of unity of $K$ are precisely $\pm 1,\pm i$.

In this case, one can do the following.\footnote{Note that this paragraph is different (much shorter) in the published version of the paper. We find that it is worth to give more explanation at this point.} First note that every polynomial $R(x)\in {\mathbb Q}(i)[x]$ can be written as $R(x)=P(x)+iQ(x)$ with $P,Q\in {\mathbb Q}[x]$. With this notation, put $\bar{R}(x)=P(x)-iQ(x)$ and $L^*(R)=L(P)+L(Q)$.

Let $\varepsilon$ be a fundamental unit of $K$. Then $K={\mathbb Q}(\varepsilon)$, and $K$ is a quadratic extension of ${\mathbb Q}(i)$. As $\varepsilon\notin {\mathbb Q}(i)$, $\varepsilon$ is a quadratic element over ${\mathbb Q}(i)$. Let $f^*(x)$ be the minimal polynomial of $\varepsilon$ over ${\mathbb Q}(i)$. Then the minimal polynomial $f(x)$ of $\varepsilon$ over $\mathbb Q$ is $f(x)=f^*(x)\bar{f^*}(x)$.
From \eqref{ujegy} we infer that
\begin{equation}
\label{eqq1}
P(\varepsilon)+iQ(\varepsilon)=0
\end{equation}
with some $P,Q\in {\mathbb Z}[x]$ and $L^*(P+iQ)=k$. Note that thus we have $L(P)+L(Q)=k$. Further, equation \eqref{eqq1} implies that
\begin{equation}
\label{eqq2}
P(x)+iQ(x)=f^*(x)g^*(x)
\end{equation}
holds with some $g^*(x)\in {\mathbb Z}[i][x]$. Letting $g(x)=g^*(x)\bar{g^*}(x)$ (which is in ${\mathbb Z}[x]$) this implies
$$
P(x)^2+Q(x)^2=f(x)g(x).
$$
As by the well-known and trivial inequalities $L(P^2)\leq L(P)^2$ and $L(Q^2)\leq L(Q)^2$ we have
$$
L(P^2+Q^2)\leq k^2,
$$
using Lemma \ref{mignotte} we get that $L(f)\leq 16k^2$.

Thus Lemma \ref{filaseta} implies that $L(g)<C_1(k)$, where $C_1(k),C_2(k),C_3(k)$ denote explicitly computable constants depending only on $k$.

The following observation will be of great help: for any $u,v,w\in {\mathbb Z}[i][x]$ and $A>0$ we have
$$
L^*(w(x)\cdot(u(x)+x^{A+\deg w+\deg u}v(x)))=L^*(w(x)\cdot(u(x)+x^{1+\deg w+\deg u}v(x))).
$$
Thus for any $g^*(x)$ with $L^*(f^*g^*)\leq k$ there exists a $g^*_0(x)\in{\mathbb Z}[i][x]$ for which $L^*(f^*g^*_0)\leq k$, with $L^*(g^*_0)=L^*(g^*)$ and $\deg g^*_0 < C_2(k)$. This follows by noting that the number of non-zero coefficients of $g^*$ is bounded by $L^*(g^*)$, and further, by the above observation, (inductively) all the 'large gaps' among consecutive non-zero coefficients of $g^*(x)$ (in view of $\deg f^*=2$) can be 'shortened' below an effectively computable bound. So we can restrict our attention to polynomials $g^*(x)$ with degree bounded in terms of $k$; in what follows, we assume that $\deg g^* < C_2(k)$.

The upper bounds established for $L(f)$ and $L(g)$ yield
$$
\max\{L^*(f^*),L^*(g^*)\}<C_3(k).
$$
This follows from the fact that for any $h(x)\in {\mathbb Z}[i][x]$, $L^*(h)$ can be explicitly bounded from above in terms of $L(h\bar{h})$ and $\deg h$ (see Theorem 2 of Mignotte \cite{mi}).

As clearly all such $f^*(x)$ and $g^*(x)$ can be explicitly listed, we can effectively check the finitely many candidate number fields ${\mathbb Q}(\varepsilon)$ defined by $f(x)$, whether \eqref{eqq1} may hold for them (where $P(x)+iQ(x)$ is defined by \eqref{eqq2}) or not.
\end{proof}

\begin{proof}[Proof of Theorem \ref{thm5}] For an integer $A$ set
$$
f_A(x)=x^d+2A^2x+2.
$$
Then by Eisenstein's theorem $f_A(x)$ is irreducible over $\Q$. Let $\alpha$ be a zero of $f_A(x)$, and put $K_A=\Q(\alpha)$. Observe that $N_{K_A/\Q}(\alpha)=2(-1)^d$. Hence every algebraic integer in $K_A$ is congruent to one of $0,1$ modulo $\alpha$. Consequently, any unit of $K_A$ is congruent to $1$ modulo $\alpha$. This immediately shows that a sum of odd number of units cannot be zero; in other words, $\od(K_A)=\infty$.

It remains to show that there are infinitely many number fields $K_A$ of the above form. If there existed only finitely many number fields of the form $K_A$, then letting $K$ be a number field containing all of them, we would obtain that, for every integer $A$, the polynomial $f_A(x)$ would have a zero in $O_K$, the ring of integers of $K$. However, it is easy to see that the algebraic curve
$$
X^d+2XY^2+2=0
$$
is non-rational. Hence, by Siegel's theorem \cite{siuj}, the set of points $(x,y)\in O_K\times \Z$ on this curve is finite which yields a contradiction.
\end{proof}

\begin{proof}[Proof of Theorem \ref{thm6}] Let $a_1=0$, $a_2=2$ and $a_3,\dots,a_{d-1}$ be fixed integers with $2<a_3<\dots<a_{d-1}$. (When $d=3$, we have only $a_1$ and $a_2$.) For any integer $N>a_{d-1}$, set
$$
f_N(x)=x(x-2)(x-a_3)\dots (x-a_{d-1})(x-N)-1.
$$
Then $f_N(x)$ is irreducible (see Westlund \cite{we} or Fl\"ugel \cite{fl}, or for more general results e.g. Gy\H{o}ry and Rim\'an \cite{gyr} or Gy\H{o}ry, Hajdu and Tijdeman \cite{gyht0} and the references there). Let $\xi_N$ be a zero of $f_N(x)$, and $K_N=\Q(\xi_N)$. Observe that $-\xi_N$ and $\xi_N-2$ are both (non-rational) units of $K$. Hence $(-\xi_N,\xi_N-2,1,1)$ is a solution of the unit equation
\begin{equation}
\label{neweqgyk}
\varepsilon_1+\varepsilon_2+\varepsilon_3+\varepsilon_4=0
\end{equation}
in units $\varepsilon_1,\varepsilon_2,\varepsilon_3,\varepsilon_4$ of $K_N$ such that $\varepsilon_3=\varepsilon_4=1$, $\varepsilon_1,\varepsilon_2$ are not rational, and the left hand side of \eqref{neweqgyk} has no vanishing subsum. This proves that $\ev(K)=4$.

It remains to show that there are infinitely many distinct number fields $K_N$ of the above type. Suppose, on the contrary, that there exist only finitely many distinct number fields $K_N$ with the above properties. Then there are infinitely many number fields $K_{N'}=\Q(\xi_{N'})$ which coincide with $K_N=\Q(\xi_N)$ for a fixed $\xi_N$. Here $\xi_{N'}$ denotes a zero of $f_{N'}(x)$.

The tuple $(-\xi_{N'},\xi_{N'}-2,1,1)$ is also a solution of \eqref{neweqgyk} for every $N'$ under consideration. But the tuples $(-\xi_N,\xi_N-2,1,1)$ and $(-\xi_{N'},\xi_{N'}-2,1,1)$ coincide only if $\xi_{N'}=\xi_N$, when $N'=N$. Consequently, equation \eqref{neweqgyk} has infinitely many distinct solutions $(-\xi_{N'},\xi_{N'}-2,1,1)$ in $K_N$, which contradicts the finiteness results of Evertse \cite{ev} and van der Poorten and Schlickewei \cite{ps} on unit equations.
\end{proof}

\noindent
{\bf Remark.} In the above proofs of Theorem \ref{thm5} and \ref{thm6} we could also use Hilbert's Irreducibility Theorem (see e.g. \cite{schuj} Theorem 46) to prove the irreducibility of $f_A(x)$ and $f_N(x)$. Further, the argument used in the second part of the proof of Theorem \ref{thm5} could also be applied at the end of the proof of Theorem \ref{thm6} as well.

\section{An application to arithmetic graphs - representing cycles}

Let $K$ be an algebraic number field, and let $A=\{\alpha_1,\dots,\alpha_m\}$ be a finite ordered subset of $O_K$, the ring of integers of $K$. Denote by ${\mathcal G}(A)$ the graph with vertex set $A$ whose edges are the pairs $[\alpha_i,\alpha_j]$ with
$$
\alpha_i-\alpha_j\in O_K^*,
$$
where $O_K^*$ denotes the unit group of $O_K$. The ordered subsets of the form $A=\{\alpha_1,\dots,\alpha_m\}$ and $A'=\{\alpha_1',\dots,\alpha_m'\}$ of $O_K$ are called equivalent if $\alpha_i'=\varepsilon\alpha_i+\beta$ $(i=1,\dots,m)$ with some $\varepsilon\in O_K^*$ and $\beta\in O_K$. Clearly, in this case the graphs ${\mathcal G}(A)$ and ${\mathcal G}(A')$ are isomorphic. The concept of ${\mathcal G}(A)$ was introduced in Gy\H{o}ry \cite{gyuj1,gyuj2}. For given $m\geq 3$, there are infinitely many equivalence classes of ordered subsets $A$ with $|A|=m$. Apart from finitely many equivalence classes, the structure of these arithmetic graphs have been described by Gy\H{o}ry; see, say \cite{gy2}. These graphs have many important applications to various and wide classes of Diophantine problems; see e.g. Gy\H{o}ry \cite{gy2}, Evertse and Gy\H{o}ry \cite{egy} and the references given there.

Gy\H{o}ry, Hajdu and Tijdeman \cite{gyht1,gyht2} performed a systematic study of of the representability of arithmetic graphs over $\Q$ in the $S$-unit case, and over algebraic number fields, respectively. Among other things, they have generalized some results of Ruzsa \cite{r}.

Ruzsa \cite{r} described the cycles\footnote{$A=\{\alpha_1,\dots,\alpha_m\}$ forms a cycle if $\alpha_i$ and $\alpha_j$ are connected with an edge if and only if either $\{i,j\}=\{1,m\}$ or $|i-j|=1$.} which are representable by arithmetic graphs over $\Q$, using $S$-units. In this case the set of vertices $A$ is a subset of $\Z$, and $[a_i,a_j]$ with $a_i,a_j\in A$ is an edge if and only if all the prime divisors of $a_i-a_j$ belong to a fixed finite set of primes $S=\{p_1,\dots,p_s\}$. Ruzsa \cite{r} gave a complete characterization of cycles in this case, by proving that if $2\in S$ then ${\mathcal G}(A)$ contains cycles of every length $\geq 3$, while if $2\nmid S$ then ${\mathcal G}(A)$ contains cycles of every {\it odd} length $\geq 3$, but none of even length. (See also \cite{chkt} for certain related problems and results.)

Now connecting to the above mentioned result of Ruzsa \cite{r}, we completely characterize the possible lengths of cycles among the graphs ${\mathcal G}(A)$ over number fields $K$, where $A$ is a finite subset of the ring of integers of $K$.

\begin{theorem}
\label{thmcyc}
Let $K$ be an algebraic number field different from $\Q$ and the imaginary quadratic fields. Then among the graphs ${\mathcal G}(A)$
\begin{itemize}
\item[i)] there are cycles of every even length $\geq 4$,
\item[ii)] there are cycles of every odd length $\geq \od(K)$, but there are no cycles of odd length $<\od(K)$.
\end{itemize}
\end{theorem}

\begin{proof} If $\od(K)=\infty$, then there is nothing to prove about odd cycles, so throughout the proof we shall assume that $\od(K)<\infty$. (It will be clear from the proof that this assumption has no effect at all on the statement concerning even cycles.) Let $\varepsilon\in O_K^*$ such that neither of $1\pm\varepsilon$ is a unit. The existence of such units follows from deep finiteness results of Siegel \cite{si} and Lang \cite{lang0} on unit equations. However, it can be seen also in many elementary ways. For example, take a prime ideal ${\mathcal P}$ in $K$ lying above $2$, and let $\eta$
be any unit of infinite order. Then writing $n$ for the order of $\eta$ modulo ${\mathcal P}$, we have $1-\eta^n\in{\mathcal P}$ and by $1+\eta^n=2-(1-\eta^n)$, also $1-\eta^n\in{\mathcal P}$. Hence taking $\varepsilon=\eta^n$, $1\pm \varepsilon$ are not units.

Then as one can readily check, $A=\{0,1,1+\varepsilon,\varepsilon\}$ is a cycle of length $4$. Observe that the existence of an odd cycle of length $<\od(K)$ would contradict the minimality of $\od(K)$. Let now $\varepsilon_1,\dots,\varepsilon_k\in O_K^*$ with $k=\od(K)$
such that $\varepsilon_1+\dots+\varepsilon_k=0$, and let
$$
\alpha_i=\varepsilon_1+\dots+\varepsilon_i\ \ \ (i=1,\dots,k).
$$
We claim that ${\mathcal G}(A)$ with $A=\{\alpha_1,\dots,\alpha_k\}$ is a cycle. For this, first observe that $\alpha_i\neq \alpha_j$ for $1\leq i<j\leq k$. Indeed, otherwise we would have
$$
\varepsilon_{i+1}+\dots+\varepsilon_j=0,
$$
and consequently
$$
\varepsilon_1+\dots+\varepsilon_i+ \varepsilon_{j+1}+\dots+\varepsilon_k=0.
$$
However, as one of $j-i$, $i+k-j$ is odd, this would violate the minimality of $\od(K)$. Then, also observe that $[\alpha_i,\alpha_j]$ with $1\leq i<j\leq k$ is an edge in ${\mathcal G}(A)$ if and only if either $j-i=1$, or $(i,j)=(1,k)$. Indeed, assume to the contrary that $[\alpha_i,\alpha_j]$ is an edge with $1\leq i$, $i+2\leq j$ and $(i,j)\neq (1,k)$. Hence $\alpha_j-\alpha_i=\varepsilon_0\in O_K^*$. Then $\alpha_j-\alpha_i-(\alpha_j-\alpha_i)=0$ implies
$$
\varepsilon_{i+1}+\dots+\varepsilon_j-\varepsilon_0=0,
$$
whence also
$$
\varepsilon_1+\dots+\varepsilon_i+ \varepsilon_{j+1}+\dots+\varepsilon_k+\varepsilon_0=0.
$$
Similarly as above, we see that one of $j-i+1$ and $i+k-j+1$ is odd. Further, $2\leq j-i\leq k-2$ shows that $\max\{j-i+1,i+k-j+1\}<k$. This violates the minimality of $\od(K)$ once again. That is, ${\mathcal G}(A)$ is a cycle (of length $\od(K)$), indeed.

Now we prove that if ${\mathcal G}(A)$ is a cycle of length $t\geq 3$, then there exists a cycle ${\mathcal G}(A')$ of length $t+2$. This clearly finishes the proof. To prove this assertion, we adopt the construction of Ruzsa from the proof of Theorem 3.1 in \cite{r}.

Suppose that
$$
A=\{\alpha_1,\alpha_2,\dots,\alpha_{t-1},\alpha_t\}
$$
is a subset of $O_K$ such that ${\mathcal G}(A)$ is a cycle (of length $t$). Let $\varepsilon\in O_K^*$ be such that $\alpha_i+\varepsilon\neq \alpha_j$ and $\alpha_i+\varepsilon-\alpha_j\notin O_K^*$ ($i=1,\dots,t-1$, $j=1,t-1,t$ with $i\neq j$). By the already mentioned finiteness results of Siegel \cite{si} and Lang \cite{lang0} on unit equations, such an $\varepsilon$ exists. (Note that this assertion could also be proved by simpler tools.) Put
$$
A'=\{\alpha_1,\alpha_1+\varepsilon,\alpha_2+\varepsilon, \dots,\alpha_{t-1}+\varepsilon,\alpha_{t-1},\alpha_t\}.
$$
Now by the choice of $\varepsilon$, using that ${\mathcal G}(A)$ is a cycle of length $t$, we easily see that ${\mathcal G}(A')$ is a cycle of length $t+2$. Hence the proof is complete.
\end{proof}

\noindent
{\bf Remark.} The cases where $K=\Q$ or $K$ is an imaginary quadratic field, can be handled easily. The only cases that need some simple considerations are $K=\Q(i),\Q(\zeta_3)$.

As an immediate consequence of Theorems \ref{thm3} and \ref{thmcyc} we obtain

\begin{cor} For every odd $t\geq 3$ there exists a number field $K$ with the following properties: among the graphs ${\mathcal G}(A)$
\begin{itemize}
\item[i)] there are cycles of every even length $\geq 4$,
\item[ii)] there are cycles of every odd length $\geq t$, but there are no cycles of odd length $<t$.
\end{itemize}
\end{cor}

\section{Problems and open questions}

In this concluding section we list some problems and open questions, and we also give a remark about a possible continuation of our research.

\vskip.2cm

\noindent{\bf Problems and open questions.}
\begin{itemize}
\item[i)] Prove that for all $k$ of the form $k=4t^4-4t+2$
$(t\in \Z\setminus\{0,1\})$ there exists a number field $K$ with $\ell(K)=k$. (That is, prove the Conjecture after Theorem \ref{thm3}).
\item[ii)] Is it true that for any $d$ with $d\geq 2$ and $a\in \Z_{\geq 4}$ even, $b\in \Z_{\geq 3}\cup\{\infty\}$ odd, there exist infinitely many number fields such that deg$(K)=d$, $\ev(K)=a$ and $\od(K)=b$?
\item[iii)] Can we say something about the distribution of $\ell(K)\pmod{n}$, where $n\geq 3$ is an integer?
\item[iv)] Are there infinitely many totally real quadratic, cubic (both totally real and complex) and totally complex quartic fields $K$, in the latter case assuming that $K$ contains no nontrivial roots of unity, with a system of fundamental units, consisting of even units? (This question is related to Theorem \ref{newcor2} and Proposition \ref{newthm3}.)
\end{itemize}

\vskip.2cm

\noindent{\bf Remark.} In a forthcoming paper, we plan to describe the properties of the set ${\mathcal L}(K)$ of those integers $k\geq 3$ for which the unit equation \eqref{gyeq1} is solvable in units $\varepsilon_1,\dots,\varepsilon_k$ of $K$ such that the left hand side of the above equation has no proper vanishing subsum. If $\ell(K)<\infty$ then $\ell(K)$ is the minimal element of ${\mathcal L}(K)$. It is clear that if $K$ is different from $\Q$ and the imaginary quadratic fields, then the set ${\mathcal L}(K)$ contains arbitrarily large values $k$. Indeed, take an arbitrary unit $\varepsilon$ in $K$, with minimal monic polynomial $f(x)=x^n+a_{n-1}x^{n-1}+\dots+a_1x+a_0$. Observe that $f(\varepsilon)=0$ can be considered as an equation of the form \eqref{gyeq1} with $k=k_\varepsilon=1+|a_{n-1}|+\dots+|a_0|$ terms, clearly with no proper vanishing subsums. Since $k_\varepsilon<C$ can be valid only for finitely many $\varepsilon$ for any constant $C$, but $K$ contains infinitely many units, $|{\mathcal L}(K)|=\infty$ follows. We (at least some of us) intend to study ${\mathcal L}(K)$ further.

\end{document}